\documentclass{proc-l}

\usepackage{amssymb}

\newtheorem{thm}{Theorem}[section]
\newtheorem{conj}[thm]{Conjecture}
\newtheorem{lem}[thm]{Lemma}
\newtheorem{prop}[thm]{Proposition}
\theoremstyle{definition}

\theoremstyle{remark}
\newtheorem{rem}[thm]{Remark}

\numberwithin{equation}{section}

\newcommand{\integers}{\mathbb{Z}}
\newcommand{\naturals}{\mathbb{N}}
\DeclareMathOperator{\scal}{scal}

\DeclareMathOperator{\pr}{\mathrm{pr}}

\begin{document}
\title{A Counterexample to a Conjecture about Positive Scalar Curvature}

\author{Daniel Pape}
\address{Georg-August-Universit\"{a}t G\"{o}ttingen,
 Bunsenstr. 3, 37073 G\"{o}ttingen, Germany}
\email{pape@uni-math.gwdg.de}
\thanks{Daniel Pape was supported by the German Research Foundation (DFG) through
the Research Training Group 1493 \textquoteleft Mathematical structures in modern quantum physics\textquoteright. \texttt{www.uni-math.gwdg.de/pape}}

\author{Thomas Schick}
\address{Georg-August-Universit\"{a}t G\"{o}ttingen, Bunsenstr. 3, 37073 G\"{o}ttingen, Germany}
\email{schick@uni-math.gwdg.de}
\thanks{Thomas Schick was partially funded by the Courant Research Center \textquoteleft Higher
  order structures in Mathematics\textquoteright\ within the German initiative
  of excellence. \texttt{www.uni-math.gwdg.de/schick}}


\date{Feb 1, 2011}

\commby{Daniel Ruberman}

\maketitle

\begin{abstract}
 \cite[Conjecture 1]{MR2587446} asserts that a closed smooth manifold $M$ with
non-spin universal covering admits a metric of positive scalar curvature if
and only if a certain homological condition is satisfied. We present a
counterexample to this conjecture, based on the 
counterexample to the unstable Gromov-Lawson-Rosenberg conjecture given in
\cite{MR1632971}.
\end{abstract}

\section{The Result}

We give a counterexample to  the following conjecture stated by Chang as
\cite[Conjecture 1]{MR2587446}, and attributed there to Rosenberg and
Weinberger. 

\begin{conj}
\label{conj:Chang's conjecture}
Suppose that $M$ is a compact oriented manifold such that its universal
covering does not admit a spin structure, with fundamental
group $\Gamma$ and of dimension $n\ge 5$. Let $f\colon M\to \underline{B}\Gamma$
be the 
composition of the classifying map $c\colon M\to B\Gamma$ of the universal
covering of $M$, and the natural map 
$B\Gamma\to\underline{B}\Gamma$. Denote by $[M]$ the fundamental class of $M$
in $H_n(M)$.  Then $M$ admits a metric of positive scalar curvature if and
only if $f_*[M]$ vanishes in  $H_n(\underline{B}\Gamma)$. 
\end{conj}

Here $B\Gamma$ is the classifying space for the group $\Gamma$
and $\underline{B}\Gamma$ is the quotient of the universal space for proper
actions, 
i.e.~the quotient $\underline E\Gamma/\Gamma$, where $\underline E\Gamma$ is a
proper $\Gamma$-space such that for every finite subgroup $F\leq \Gamma$ the
fixed point set $\underline E\Gamma^F$ is contractible (in particular,
non-empty), but such that $\underline E\Gamma^H=\emptyset$ for all other
subgroups $H\leq \Gamma$, compare
\cite[p.~1623]{MR2587446}. 

\medskip

Our counterexample is based on the counterexample to the
Gromov-Lawson-Rosenberg conjecture given in \cite{MR1632971}. There, a
5-dimensional connected closed spin manifold $M$ with fundamental group
$\Gamma=\integers^4\oplus\integers/3$ is constructed, whose Rosenberg index vanishes but
which nevertheless does not admit a  metric of positive scalar curvature. By
taking the connected sum of this manifold $M$ with a simply-connected non-spin
manifold $N$, we obtain a totally non-spin manifold $X$ which has the same
fundamental group as $M$. One has $B\Gamma=T^4\times B\integers/3$ and
analogously $\underline{B}\Gamma=T^4$ by \cite[(1) and (4),
p. 1624]{MR2587446}. Especially, $H_n(\underline{B}\Gamma)=0$ for $n\geq 5$,
so that the condition on $f_*[X]$ from Conjecture~\ref{conj:Chang's
  conjecture} is satisfied in the case at hand. The argument in
\cite{MR1632971} relies on the following observation by Stolz and we will also
make significant use of this result. 

\begin{lem}
\label{lem:Stolz' observation}
Let $X$ be a topological space and set for $n\in\naturals_{\geq 2}$
\[
H_n^+(X):=\{f_*[M]\in H_n(X)\,;\, \textnormal{$f\colon M^n\to X$ and $M$ admits a
  metric with $\scal>0$}\}
\]
 Then for any class $u\in H^1(X)$ the 
map
\[
u\cap\phantom{x}\colon H_n(X)\to H_{n-1}(X)\quad,\quad x\mapsto u\cap x
\]
maps $H_n^+(X)$ into $H_{n-1}^+(X)$ if $3\leq n\leq 8$.
\end{lem}
\begin{proof}
See \cite[Corollary 1.5]{MR1632971} for $3\le n\le 7$ and \cite[Theorem
4.4]{JoachimSchick} for $n=8$. 
\end{proof}

Our result reads now as follows.

\begin{prop}
\label{prop:Main result}
Let $M$ be the manifold constructed in \cite{MR1632971} (we recall its
construction in Section \ref{sec:proof}) and $N$ a simply connected manifold
of dimension 5 which admits no spin structure.
Then the manifold $X:=M\# N$ has non-spin universal covering and admits no
metric with positive scalar curvature.
\end{prop}

This result is part of the first named author's forthcoming thesis~\cite{thesis_pape}.

\section{The Proof}\label{sec:proof}

\begin{proof}[Proof of Proposition~\ref{prop:Main result}]
First of all, if $X$ is constructed as above, we have already noted that it
has non-spin universal covering. To obtain an explicit simply-connected
non-spin $5$-manifold $N$, one can start with 
$\mathbb{C} P^2\times S^1$, which is non-spin as $\mathbb{C}P^2$ is, and then
do surgery on the embedded $S^1$ to obtain the simply-connected $N$. Because
this surgery does not touch 
the embedded $\mathbb{C}P^1$ with its non-spin normal bundle, the resulting
$N$ remains a non-spin manifold.

\medskip

In order to see that $X$ admits no metric of positive scalar curvature, we use the same argument
as in \cite{MR1632971}. To begin with, we choose the model $B\Gamma=T^4\times B\integers/3$. Recall,
\[
H_n(T^d)\cong\integers^{d(n)}\quad,\quad d(n)=\binom{d}{n}
\]
and
\[
H_n(B\integers/k\integers)\cong
\begin{cases}
\integers,   & \textnormal{$n=0$;}   \\
\integers/k\integers, & \textnormal{$n$ odd;} \\
0,      & \textnormal{$n$ even.}
\end{cases}
\]
Together with the K\"{u}nneth formula this gives
\[
H_k(B\Gamma)=\bigoplus_{p_1+\cdots+p_5=k}
H_{p_1}(X_1)\otimes\cdots\otimes  H_{p_5}(X_5)
\text{ .}
\]
Here we have written $T^4=X_1\times\cdots\times X_4$ as product
of four copies of $S^1$, and $X_5$ for $B\integers/3$. 

Fix a basepoint $x=(x_1,\ldots,x_5)\in B\Gamma$ and let
$p\colon T\to B\integers/3$ be a map which induces an epimorphism on $\pi_1$ as in
\cite{MR1632971}, as well as 
$f_j\colon X_j\to B\Gamma$ the map which includes $X_j$ identically
and basepoint-preserving. We denote by $[*]\in H_0(B\Gamma)$ the canonical
generator.
Next, choose for each $1\leq j\leq 4$ generators $g_j\in H_1(X_j)$
and elements $g_j^*\in H^1(X_j)$ with $\langle g_j^*,g_j\rangle=1$, and let
$g_5\in H_1(X_5)$ be $p_*[S^1]$ where $[S^1]$ is the standard generator for
$H_1(S^1)$. 
Introduce the  elements $v_j:= (f_j)_*(g_j)\in H_1(B\Gamma)$ for $j=1,\dots,
5$ as well as $a_1,\dots, a_4\in H^1(B\Gamma)$ with
\begin{align*}
a_1 &:= (\pr_1)^*(g_1^*)\times 1               \times 1    \times 1    \times 1\text{ ,}\\
a_2 &:= 1               \times (\pr_2)^*(g_2^*) \times 1    \times 1    \times 1\text{ ,}\\
a_3 &:= 1               \times 1               \times (\pr_3)^*(g_3^*) \times 1\times 1\text{ ,}\\
a_4 &:= 1               \times 1               \times 1                \times (\pr_4)^*(g_4^*)\times 1\text{ .}\\
\end{align*}
Finally, set
\[
w:=v_1\times\cdots\times v_4\times v_5\in H_5(B\Gamma)
\]
and
\[
z:=[*]\times[*]\times[*]\times v_4\times v_5\in H_2(B\Gamma)\text{ .}
\]
By the K\"unneth formula, $w\neq 0$ and $z\neq 0$. Furthermore,
\[
\label{eq:z_cap}
z=a_1\cap(a_2\cap(a_3\cap w))\in H_2(B\Gamma)\text{ .}
\tag{$\ast$}
\]
For example one has
\begin{align*}
a_3\cap w &= \big(\big(1\times 1\times (\pr_3)^*(g_3^*)\big)\times \big(1\times 1\big)\big)
\cap\big(\big(v_1\times v_2\times v_3\big)\times \big(v_4\times v_5\big)\big) \\
&=
\big(\big(1\times 1\times (\pr_3)^*(g_3^*)\big)\cap \big(v_1\times v_2\times v_3\big)\big)
\times
\big((1\times 1)\cap (v_4\times v_5)\big) \\
&=
\big(\big(1\cap v_1\big)\times\big(1\cap v_2\big)\times \big((\pr_3)^*(g_3)^*\cap v_3\big)\big)
\times\big(\big(1\cap v_4\big)\times \big(1\cap v_5\big)\big) \\
&=
v_1\times v_2\times \big((\pr_3)^*(g_3^*)\cap v_3\big)\times v_4\times v_5 \\
&=
v_1\times v_2\times [*]\times v_4\times v_5\text{ ,}
\end{align*}
because of $(\pr_3)^*(g_3^*)\cap (i_3)_*(g_3)=\langle g_3^*,g_3\rangle [*]$.
Let $f\colon T^5\to T^4\times B\integers/3$ be given by $f=(f_1\times
f_2\times 
f_3\times f_4)\times (f_5\circ p)$ and choose $(g_1\times\cdots\times
g_4)\times [S^1]=:[T^5]$ as fundamental class for $T^5$. Then
$f_*[T^5]=w$. As in \cite{MR1632971} one can construct a bordism 
in $\Omega_5^{\mathrm{spin}}(B\Gamma)$ from $f$ to a map $g\colon M\to
B\Gamma$ which induces an isomorphism of fundamental groups. This defines the
manifold $M$. Now let $N$ be any simply-connected closed non-spin manifold of
dimension $5$ and set $X:=M\# N$. 

Finally, assume that $X$
admits a metric of positive scalar curvature. Then consider the map $h\colon
M\sqcup N\to B\Gamma$  
on the disjoint union of $M$ and $N$, which equals $g$ on $M$ and sends $N$ to
a point. One has $h_*[M\sqcup N]=g_*[M]=w$ and since $M\sqcup N$ is bordant to
$M\# N$, it follows 
that $w\in H^+_5(X)$. But then it follows from~\eqref{eq:z_cap} as well as
Lemma~\ref{lem:Stolz' observation} that $w$ is mapped to $z$ under the
following composition
\begin{equation*}
H_5^+(B\Gamma) \xrightarrow{a_3\cap\,\cdot\,}{} 
H_4^+(B\Gamma) \xrightarrow{a_2\cap\,\cdot\,}{} 
H_3^+(B\Gamma) \xrightarrow{a_1\cap\,\cdot\,}{} 
H_2^+(B\Gamma) .
\end{equation*}
Hence $z=k_*[S^2]$ for some $k\colon S^2\to B\Gamma$ since
$S^2$ is the only orientable surface which admits a metric of positive scalar
curvature. On the other hand, $\pi_2(B\Gamma)=0$ so that $k$ is null homotopic.
This implies $z=0$, which is a contradiction. 
\end{proof}

\begin{rem}
  The method described in this note produces a counterexample to Conjecture
  \ref{conj:Chang's conjecture} with fundamental group $\Gamma$ whenever
    $\Gamma$ satisfies the following  homological conditions:
    \begin{itemize}
    \item for $5\le m\le 8$ there is a homology class $[M]\in
      H_m(B\Gamma;\integers)$ 
      represented by an $m$-dimensional closed oriented manifold $M$ (with
      surgeries one can then arrange that $\pi_1(M)=\Gamma$)
    \item there are classes $\alpha_1,\alpha_{m-2}\in H^1(B\Gamma;\integers)$
      such that $\alpha_1\cap(\dots \cap (\alpha_{m-2}\cap [M])) \ne 0 \in
      H_2(B\Gamma;\integers)$
    \item under the map $H_m(B\Gamma)\to H_m(\underline{B}\Gamma)$ the class
      $[M]$ is send to $0$.
    \end{itemize}
  Note that this condition is similar, indeed much easier than the general
  homological 
  condition for counterexamples to the Gromov-Lawson-Rosenberg condition
  derived in \cite{MR2048721}. Unfortunately, its structure requires the group
  $\Gamma$ to contain non-trivial torsion, to allow for a kernel of the map
  $H_*(B\Gamma)\to H_*(\underline{B}\Gamma)$ (in contrast to
  \cite{MR2048721}).

  The assumption on $H^1(B\Gamma;\integers)$ we have to make is very strong,
  it has to have rank at least $m-2$. In particular, the method tells us
  nothing about finite groups. Indeed, the question of
  existence for $k$ or lower dimensional manifolds with finite fundamental
  group $(\integers/p\integers)^k$ for $p$ odd is completely open (in the
  totally non-spin case
  as well as in the spin case) and seems the first obstacle for a full
  understanding of this problem. Progress in this direction will require
  a completely new set of ideas.
\end{rem}

\addcontentsline{toc}{section}{References}         %
\bibliographystyle{amsplain}
\bibliography{Literaturdatenbank_Counterexample_Chang}
\end{document}